\documentclass[12pt,reqno]{amsart}
 \usepackage[utf8]{inputenc}
 \usepackage[T1]{fontenc}
 \usepackage[usenames, dvipsnames]{color}
 \usepackage{ulem}
 
 \usepackage{dsfont, amsfonts, amsmath, amssymb,amscd, stmaryrd, latexsym, amsthm, dsfont}
 \usepackage[frenchb,english]{babel}
 \usepackage{enumerate}
 \usepackage{longtable}
 \usepackage{geometry}
 \usepackage{float}
 \usepackage{tikz}
 \usetikzlibrary{shapes,arrows}
 \usepackage{float}
 \usepackage{tikz}
 \usepackage{xypic}
 \usetikzlibrary{shapes,arrows}
 
 \theoremstyle{plain}

 \newtheorem{theorem}{Theorem}[section]
 \newtheorem{lemma}[theorem]{Lemma}
 \newtheorem{proposition}[theorem]{Proposition}
 \newtheorem{corollary}[theorem]{Corollary}

 \theoremstyle{remark}

 \newtheorem{definition}[theorem]{\bf Definition}

 \def\QQ{\mathbb{Q}}

 \usepackage{hyperref}
 \hypersetup{
 	colorlinks=true,
 	urlcolor=blue,
 	citecolor=blue}
 
\begin{document}

 	\selectlanguage{english}
 	\title[ Hilbert genus fields ...]{ Hilbert genus fields of some number fields with high degrees}

 		\author[M. M. Chems-Eddin]{Mohamed Mahmoud Chems-Eddin}
 	\address{Mohamed Mahmoud Chems-Eddin: Sidi Mohamed Ben Abdellah University,   
 		Faculty of Sciences Dhar El Mahraz, Departement of Mathematics, Fez, Morocco }
 	\email{2m.chemseddin@gmail.com}
 	
 	\author[M. A. Hajjami]{Moulay Ahmed Hajjami} 
 	\address{Moulay Ahmed Hajjami: Moulay Ismail University of Meknes, Faculty of  Sciences and Technology, Department of  Mathematics, Errachidia, Morocco.}
 	\email{a.hajjami76@gmail.com}

 	\author[M. Taous]{Mohammed Taous}
 	\address{Mohammed Taous: Moulay Ismail University of Meknes, Faculty of Sciences, Departement of Mathematics, Meknes,   Morocco.}
 	\email{taousm@hotmail.com}

 	
 	\begin{abstract}
   The aim of this paper is to give some properties of Hilbert genus fields and 
	  construct  the Hilbert genus fields of the fields $L_{m,d}:=\mathbb{Q}(\zeta_{2^m},\sqrt{d})$, where      $m\geq 3$ is a positive integer and $d$   is a square-free integer whose    prime divisors   are congruent to $\pm 3\pmod 8$ or 
$9\pmod{16}$.
 	\end{abstract}
 	
 	\selectlanguage{english}
 	
 	\maketitle
 	
 	\section{\bf Introduction}\label{sec:1}
Let $d$ be a square-free integer and  $m\geq 3$ a positive integer. The fields 	$L_{m,d}:=\mathbb{Q}(\zeta_{2^m},\sqrt{d})$  represent the layers of the cyclotomic $\mathbb{Z}_2$-extension of the special Dirichlet fields $\mathbb{Q}(\sqrt{-1},\sqrt{d})$ and they were the subject of some recent studies, which gave interest to their $2$-class groups (\textnormal{cf}. \cite{Chemstherank,chemskatharina}).  In the present, work we shall construct the Hilbert genus fields of these fields.

 This question  have been  investigated by many mathematicians for   fields of small degrees (\textnormal{eg}. $2$ and $4$). For example,   Bae and Yue   studied the Hilbert genus field of the fields $\mathbb{Q}(\sqrt{p}, \sqrt{d})$, for 
 a prime number $p$ such that, $p=2$ or $p\equiv 1\pmod 4$, and    a positive square-free  integer $d$ (\textnormal{cf}. \cite{ref2}).

 Thereafter, Ouyang and Zhang have determined the Hilbert genus field of the imaginary biquadratic fields $ \mathbb{Q}(\sqrt{\delta}, \sqrt{d})$, where $\delta =-1, -2$  or $-p$ with $p\equiv 3\pmod 4$ a prime number and $d$ any square-free integer. Thereafter,  they   constructed   the Hilbert genus field of real biquadratic fields $\mathbb{Q}(\sqrt{\delta}, \sqrt{d})$, 
 for   any  positive square-free integer $d$, and   $\delta=p, 2p$ or $p_1p_2$ where $p$, $p_1$ and $p_2$ are prime numbers congruent to $3\pmod 4$, such that the class number of $ \mathbb{Q}(\sqrt{\delta})$ is odd (\textnormal{cf}.  \cite{OuyangZhang2014,OuyangZhang2015}). For more works on this subject, we refer the reader to the  papers \cite{OuyangZhang2015,ref19,hajjaChems,ref20}.

 In this paper, we shall construct the Hilbert genus field  of the fields $L_{m,d}$, $m\geq 3$, whenever all the prime divisors of $d$ are congruent to $\pm 3\pmod 8$ or 
 $9\pmod{16}$. Our results generalize some results of Ouyang and Zhang on the fields $\mathbb{Q}(\sqrt{-1},\sqrt{d})$ (\textnormal{cf}.  \cite{OuyangZhang2014}).
 
 The plan of this paper is as follows; in Section \ref{sec:2}, we   introduce   some  properties of the Hilbert genus field. In the last section,  we give the list of Hilbert genus fields of the fields  $L_{m,d}$.
 	
 	Let us stick the following notations: Let $k$ be a number field. $\mathcal{O}_k$   the ring of integers      $k$.
 		$E(k)$ the Hilbert genus field of $k$. $N_{k/k'}$ denotes  the norm map of an extension $k/k'$.   Let $r_2(\mathrm{G})$ denote the $2$-rank of an abelian finite group $\mathrm{G}$. Denote by $\zeta_{n}$  an  $n$th primitive root of unity.
 	For more notations see   the beginning of each section below.

 	\section{\bf Some results on Hilbert genus fields}\label{sec:2}

 	  	Let us  start this section by    recalling  some definitions and results.	
 	Let $k$ be a number field and denote by $\mathbf{C}l(k)$ its class group. It is know by class field theory that we have  $\mathrm{G}:=\mathrm{Gal}(H(k)/k)\simeq \mathbf{C}l(k)$, where 
 	$ H(k)$  is the Hilbert class field. The 
 	Hilbert genus field of $k$   is defined as the invariant field $E(k)$ of $\mathrm{G}^2$. Then,   we  have:
 	$$\mathbf{C}l(k)/\mathbf{C}l(k)^2  \simeq \mathrm{G}/\mathrm{G}^2 \simeq \mathrm{Gal}(E(k)/k),$$
 	and thus,   $r_2(\mathbf{C}l(k))$ =   $r_2(\mathrm{Gal}(E(k)/k))$. On the other hand, $E(k)/k$ is the maximal unramified Kummer extension of exponent $2$. Thus, by Kummer theory (cf. \cite[p. 14]{ref16}), there exists a unique multiplicative group $\Delta(k)$ such that
 	$$
 	E(k)=H(k)\cap k(\sqrt{k^*})= k(\sqrt{\Delta(k)})  \text{  and }  {k^*}^2 \subset {\Delta(k)} \subset k^*.
 	$$
 	Therefore, the  construction of   the Hilbert genus field of $k$ is equivalent  to the determination of  a set of generators for the finite group $\Delta(k)/k{^*}^2$.
 	\vspace{0.3cm}

We need the following proposition which is a particular case of the result in \cite[p. 239]{Mollin}.
 	
 	\begin{proposition}[\cite{Mollin}, Theorem 5.20]\label{unramifed quad ex}
 		Let $k/k'$ be a quadratic extension of number fields  and $\alpha$ a element of $k'$, coprime with $2$, such that $k=k'(\sqrt{\alpha})$. The extension $k/k'$ is unramified  at all finite primes of $k'$ if  and only  if     the two following items  hold:
 		\begin{enumerate}[\indent\rm 1.]
 			\item the ideal generated by $\alpha$ is the square of the fractional ideal of $k'$, and 
 			\item there exists a nonzero number $\xi$ of $k'$ verifying $\alpha \equiv \xi^2\pmod{4}$.
 		\end{enumerate}
 	\end{proposition}
 	
 	\begin{corollary}\label{primitive unramified extensions}
 		Let $k/k'$ be an extension of number fields and let $\alpha\in k'$ be a square-free in $k$. If $k'(\sqrt{\alpha})/k'$ is unramified, then $k(\sqrt{\alpha})/k$ is also unramified.
 	\end{corollary}
 	\begin{proof}
 		As $k'(\sqrt{\alpha})/k'$ is unramified, then there exist an ideal $\mathfrak b$ of $\mathcal{O}_{k'}$ such that $\alpha\mathcal{O}_{k'}=\mathfrak b^2$ and  $\xi\in k'$ verifying $\alpha \equiv \xi^2\pmod{4}$.
 		Therefore, $\alpha\mathcal{O}_{k}=(\mathfrak b\mathcal{O}_{k})^2$ and  $\xi\in k'$ verifying $\alpha \equiv \xi^2\pmod{4}$. So the result by
 		Proposition \ref{unramifed quad ex}.
 	\end{proof}
 	
 	Let us put the following definition.
 	
 	\begin{definition}
 		A number field $k$  is said QO-field if it is a quadratic extension of certain number field $k'$ whose class number is odd. We shall call $k'$ a base field of the QO-field $k$ and the $k/k'$ is a QO-extension. 	\end{definition}

 		\begin{lemma}\label{lemma 4.1}
 		Let $k/k'$ be   a QO-extension.   Let $\Delta(k)$ denote the multiplicative group  such that ${k^*}^2 \subset {\Delta(k)} \subset k^*$ and $k(\sqrt{\Delta(k)})$   the Hilbert genus field of $k$  $($see  the beginning of this section$)$. Then 
 		$$r_2(\Delta(k)/{k^*}^2)= t-e-1,$$
 		  	with  $t$ is the number of  ramified primes (finite or infinite) in the extension  $k/k'$ and $e$ is  defined by   $2^{e}=[E_{k'}:E_{k'} \cap N_{k/k'}(k^*)]$, where $E_{k'}$ is the unit group of $k'$. 
 	\end{lemma}
 	\begin{proof}Put  $ \mathrm{G}=\mathrm{Gal}(H(k)/k)$.
 		By the definition of $E(k) $   $($see  the beginning of this section$)$ we have: $$\mathbf{C}l(k)/\mathbf{C}l(k)^2   \simeq \mathrm{G}/\mathrm{G}^2 \simeq \mathrm{Gal}(E(k)/k),$$
 		Since $E(k)=k(\sqrt{\Delta(k)})$, then by class field theory:
 		$$r_2(\Delta(k)/{k^*}^2)= \log _2[k(\sqrt{\Delta(k)}): k]=\log _2[E(k):k]=r_2(\mathbf{C}l(k)).$$
 		So the result by   the well known ambiguous class number formula     (cf.    \cite[Lemma 2.4]{yueRedeix}).	\end{proof}

 	\begin{theorem}\label{thm gns fld of  QO}
 		Let $k=k'(\sqrt{\mu})$ be a ramified quadratic extension  of QO-fields that have the same base field. Assume furthermore, that $r_2(\mathbf{C}l(k))=r_2( \mathbf{C}l(k') )$.
 		Then we have $$E(k)=k(E(k')).$$
 	\end{theorem}
 	\begin{proof}
 		Let $\gamma_1$, ..., $\gamma_r$ be $r=r_2( \mathbf{C}l(k') )$ elements of $k'$ such that,
 		$$E(k')=k'(\sqrt{\gamma_1}, ..., \sqrt{\gamma_r}).$$
 		By Corollary \ref{primitive unramified extensions}, $k(\sqrt{\gamma_i})/k$ is unramified  for all $i$. Therefore, $\{ {\gamma_1}, ...,  {\gamma_r}\}$ is   a representative set of 
 		$\Delta(k)/k^2$. In fact, it suffices to prove that it is linearly independent  modulo ${k^*}^2 $. Put $\beta={\gamma_1} {\gamma_2}^{a_2}  ...   {\gamma_r}^{a_r}$, for some $a_i\in\{0,1\}$, and assume that $\beta \in k^2$. One can easily check that we have  $\beta\in k'^2$ or 	$\mu\beta\in k'^2$.
 		\begin{enumerate}[\indent\rm $\bullet$]
 			\item Since $\{ {\gamma_1}, ...,  {\gamma_r}\}$ is   a representative set of 
 			$\Delta(k')/k'^2$, the  case  $\beta\in k^2$,  is clearly impossible.
 			\item Assume that $\mu\beta\in k'^2$. Thus, $\mu{\gamma_1} {\gamma_2}^{a_2}  ...   {\gamma_r}^{a_r}= \alpha^2$, for some $\alpha \in k'$. 
 			Since $k/k'$ is ramified, this equality implies that there is a square-free element $\mu'$ of $k'$ such  that $\mu'=\alpha'^2$, for some $\alpha' \in k'$, which is impossible.
 		\end{enumerate}	
 		\noindent Since  $r_2(\mathbf{C}l(k))=r$, we have the result.	
 	\end{proof}

 	Thus, we have the following corollary.
 	\begin{corollary}\label{useful cor} 		
 		Let $k$ be a QO-number field, $k'$ its base field and $k_n$ be the $n$th layer of its cyclotomic $\mathbb{Z}_2$-extension. Let $n_0$ be an integer such that every prime which ramify in $k_\infty/k_{n_0}$ is totally ramified and $r_2(\mathbf{C}l(k_{n_0}))=r_2( \mathbf{C}l(k_{n_0+1}) )=r$.  Assume furthermore that   $k_{n_0+1}$ is a QO-field whose base field is the $(n_0+1)$th layer of the cyclotomic $\mathbb{Z}_2$-extension   of $k'$.
 		If $E(k)=k_{n_0}(\sqrt{\gamma_1}, ..., \sqrt{\gamma_r}) $ and $k_{n_0}'$ has  odd class number,  then 
 		$$ E(k_n)=k_n(\sqrt{\gamma_1}, ...,\sqrt{\gamma_r}),$$
 		for all $n\geq n_0$. Furthermore,  $ E(k_\infty)=k_\infty(\sqrt{\gamma_1}, ..., \sqrt{\gamma_r}).$ 
 	\end{corollary}
  	\begin{proof}
Under the hypothesis of the corollary, one can use the well known result of Fukuda (\textnormal{cf}.  \cite[Theorem 1]{fukuda}) to check what follows:
\begin{enumerate}[\indent$\bullet$]
	\item The class number of $k'_n$, the $n$th layer of  cyclotomic $\mathbb{Z}_2$-extension   of $k'$, is odd for all $n$.
	\item $r_2(\mathbf{C}l(k_{n}))=r$, for all $n\geq n_0$.
\end{enumerate}
 Therefore, a recurrent application of Theorem \ref{thm gns fld of  QO} gives the corollary.
  	\end{proof}

 	\section{\bf The Hilbert genus field of  $L_{m,d}=\mathbb{Q}(\zeta_{2^m},\sqrt{d})$}\label{sec:3}	
 	
 	Keep notations fixed in the previous sections. Now we  state the main result of the paper. For the existance of solutions of the diophantine equations mentioned on the theorem below, one can see \cite{OuyangZhang2014} and \cite[p. 324]{kaplan76}.
 	\begin{theorem}
 	Let  $d$ be a   square-free integer whose prime divisors are congruent to  $\pm 3\pmod 8$ or $9\pmod {16}$. Let $n$ be the number of prime divisors of    $d$,
 	 $s$ the number of prime divisors of 
 		$d$ which are  congruent $5\pmod 8$, $t$ the number of prime divisors of   
 		$d$ which are  congruent $3\pmod 8$ and  $r$ the number  of those which are congruent  $9\pmod{16}$. We have $n=r+t+s$.  Let  $d=\ell_1...\ell_r p_1...p_s q_1...q_t$
 		be the factorization of $d$ such that $\ell_i\equiv 9\pmod{16}$ and $p_i\equiv 5\equiv -q_i\pmod 8$.  Put $L_{m,d}=\QQ(\zeta_{ 2^{m}}, \sqrt{d}),$ $m\geq 3$.

 		\begin{enumerate}[\rm 1.]
 			
 			\item Assume that  $r=t=0$ and $s\geq2$. Let $(x_i,y_i)$  be a primitive solution of  
 			$p_1p_i=x^2+y^2$ $(i\geq 2)$ such that  $x_i\equiv 1\pmod 2$ and $y_i\equiv 0\pmod 4$. Then for all $m\geq 3$,
 			$$ E(L_{m,d})= L_{m,d}( \sqrt{p_1}, ..., \sqrt{p_{n-1}},\sqrt{\gamma_2},... ,\sqrt{\gamma_n}),$$
 			where $\gamma_i=x_i+y_j\sqrt{-1}$.
 			
 			\item 
 		Assume that $r=s=0$ and $t\geq2$. Let $(x_i',y_i')$  be a primitive solution of 
 			$q_1q_i=x^2+2y^2$ $(i\geq 2)$ such that $x_i'\equiv 1\pmod 2$ and $y_i'\equiv 0\pmod 2$ . Then for all $m\geq 1$,
 			$$ E(L_{m,d})= L_{m,d}( \sqrt{q_1}, ..., \sqrt{q_{n-1}},\sqrt{\alpha_2},... ,\sqrt{\alpha_n}),$$
 			where $\alpha_i=x_i'+y_i'\sqrt{-2}$.

 			\item Assume that $r=0$ and $t,s\geq2$.  Then for all $m\geq 3$,
 			\begin{eqnarray*}
 				E(L_{m,d})=  L_{m,d}
 				\left( \prod_{i=1}^{s}\QQ(\sqrt{p_i})\right)	\left( \prod_{i=1}^{t-1}\QQ(\sqrt{q_i})\right)
 				\left( \prod_{i=2}^{s}\QQ(\sqrt{\gamma_i})\right) \left( \prod_{i=2}^{t}\QQ(\sqrt{\alpha_i})\right)  ,
 			\end{eqnarray*}
 			where $\alpha_i$ and $\gamma_i$ are defined in the first two items.	
 			
 			\item  Assume that $r\geq 1$ and $t,s\geq2$. Let $(a_i,b_i)$, $(e_i,f_i)$ and $(u_i,v_i)$ such $\ell_i=a_i^2+16b_i^2=e_i^2-32f_i^2=u_i^2+8v_i^2$ and $a_i\equiv e_i\equiv u_i\equiv \pm 1\pmod{4}$. Then for all $m\geq 3$,
 			\begin{eqnarray*}
 				E(L_{m,d})= L_{m,d}   \left( \prod_{i=1}^{r}\QQ(\sqrt{l_i}, \sqrt{\pi_{1,i}}, \sqrt{\pi_{2,i}},\sqrt{\pi_{3,i}})\right) 
 				\left( \prod_{i=1}^{s}\QQ(\sqrt{p_i})\right)	\left( \prod_{i=1}^{t-1}\QQ(\sqrt{q_i})\right)\\
 				\left( \prod_{i=2}^{s}\QQ(\sqrt{\gamma_i})\right) \left( \prod_{i=2}^{t}\QQ(\sqrt{\alpha_i})\right)  ,
 			\end{eqnarray*}
 		 	where $\alpha_i$ and $\gamma_i$ are defined in the first two items, and $\pi_{1,i}=a_i+4ib_i$, $\pi_{2,i}= e_i+4f_i\sqrt{2}$ and $\pi_{3,i}= u_i+2v_i\sqrt{-2}$.
 			
 			\item Assume that $t=0$, $r\geq 1$ and $s\geq2$. For all $m\geq 3$, we have
 			\begin{enumerate}[\rm $\bullet$]
 				\item If $\exists i\in\{1, \dots, r\}$ such that  $ \left(\frac{2}{ l_i}\right)_4=1$, then
 				\begin{eqnarray*}
 					E(L_{m,d})= L_{m,d} \left( \prod_{i=1}^{r}\QQ(\sqrt{\pi_{1,i}}, \sqrt{\pi_ {2,i}},\sqrt{\pi_{3,i}})\right)
 					\left( \prod_{i=1}^{r-1}\QQ( \sqrt{l_i})\right)
 					\left( \prod_{i=1}^{s-1}\QQ(\sqrt{p_i})\right)
 					\left( \prod_{i=2}^{s}\QQ(\sqrt{\gamma_i})\right) ,
 				\end{eqnarray*}
 				\item else,
 				\begin{eqnarray*}
 					E(L_{m,d})= L_{m,d} \left( \prod_{i=1}^{r}\QQ(\sqrt{l_i}, \sqrt{\pi_{1,i}}, \sqrt{\pi_ {2,i}},\sqrt{\pi_{3,i}})\right)
 					\left( \prod_{i=1}^{s-1}\QQ(\sqrt{p_i})\right)
 					\left( \prod_{i=2}^{s}\QQ(\sqrt{\gamma_i})\right) ,
 				\end{eqnarray*}
 			\end{enumerate}
 			where $\gamma_i$ , $\pi_{1,i}$, $\pi_{2,i}$ and $\pi_{3,i}$ are defined as above.

 			\item  Assume that  $s=0$, $r\geq 1$ and $t\geq2$. For all $m\geq 3$, we have:

 			\begin{enumerate}[\rm $\bullet$]
 				\item If $\exists i\in\{1, \dots, r\}$ such that  $ \left(\frac{2}{ l_i}\right)_4=-1$, then
 				\begin{eqnarray*}
 					E(L_{m,d})= L_{m,d} \left( \prod_{i=1}^{r}\QQ(\sqrt{\pi_{1,i}}, \sqrt{\pi_ {2,i}},\sqrt{\pi_{3,i}})\right)
 					\left( \prod_{i=1}^{r-1}\QQ( \sqrt{l_i})\right)
 					\left( \prod_{i=1}^{t-1}\QQ(\sqrt{q_i})\right)
 					\left( \prod_{i=2}^{t}\QQ(\sqrt{\alpha_i})\right) ,
 				\end{eqnarray*}
 				
 				\item else,
 				\begin{eqnarray*}
 					E(L_{m,d})= L_{m,d} \left( \prod_{i=1}^{r}\QQ(\sqrt{l_i}, \sqrt{\pi_{1,i}}, \sqrt{\pi_ {2,i}},\sqrt{\pi_{3,i}})\right)
 					\left( \prod_{i=1}^{t-1}\QQ(\sqrt{q_i})\right)
 					\left( \prod_{i=2}^{t}\QQ(\sqrt{\alpha_i})\right) ,
 				\end{eqnarray*}

 			\end{enumerate}
 			where $\alpha_i$ , $\pi_{1,i}$, $\pi_{2,i}$ and $\pi_{3,i}$  are defined as above.	
 			\item Assume that $t=s=0$ and $r\geq 2$. For all $m\geq 3$, we have:
 			\begin{enumerate}[\rm $\bullet$]
 				\item If there exist $i, j\in\{1, \dots, r\}$ such that  $ \left(\frac{2}{ l_i}\right)_4\neq \left(\frac{2}{ l_j}\right)_4$, then

 				\begin{eqnarray*}
 					E(L_{m,d})=  L_{m,d} \left( \prod_{i=1}^{r}\QQ(\sqrt{l_i})\right)\left(\prod_{i=1}^{r-1}\QQ ( \sqrt{\pi_{1,i}}, \sqrt{\pi_ {2,i}},\sqrt{\pi_{3,i}})\right) ,
 				\end{eqnarray*}
 				\item else,
 				\begin{eqnarray*}
 				E(L_{m,d})=  L_{m,d} \left( \prod_{i=1}^{r}\QQ( \sqrt{l_i}, \sqrt{\pi_{1,i}})\right) \left( \prod_{i=1}^{r-1}\QQ( \sqrt{\pi_ {2,i}},\sqrt{\pi_{3,i}})\right)
 				,
 			\end{eqnarray*}
 				where $\pi_{1,i}$, $\pi_{2,i}$ and $\pi_{3,i}$ are defined as above.	
 			\end{enumerate}
 		\item Assume that $s=1$, $r\geq 1$ and $t\geq 2$.  For all  $m\geq 3$, we have
 			\begin{eqnarray*}
 			E(L_{m,d})= L_{m,d} \left( \prod_{i=1}^{r}\QQ(\sqrt{l_i}, \sqrt{\pi_{1,i}}, \sqrt{\pi_ {2,i}},\sqrt{\pi_{3,i}})\right)
 			\left( \prod_{i=1}^{t}\QQ(\sqrt{q_i})\right)
 			\left( \prod_{i=2}^{t}\QQ(\sqrt{\alpha_i})\right).
 		\end{eqnarray*}
 	
 		\item Assume that $t=1$, $r\geq 1$ and $s\geq 2$.  For all  $m\geq 3$, we have
 	\begin{eqnarray*}
 		E(L_{m,d})= L_{m,d} \left( \prod_{i=1}^{r}\QQ(\sqrt{l_i}, \sqrt{\pi_{1,i}}, \sqrt{\pi_ {2,i}},\sqrt{\pi_{3,i}})\right)
 		\left( \prod_{i=1}^{s}\QQ(\sqrt{p_i})\right)
 		\left( \prod_{i=2}^{s}\QQ(\sqrt{\gamma_i})\right).
 	\end{eqnarray*}
 		
 			\item Assume that $s=t=1$ and $r\geq 1$.  For all  $m\geq 3$, we have
 		\begin{eqnarray*}
 			E(L_{m,d})= L_{m,d} \left( \prod_{i=1}^{r}\QQ(\sqrt{l_i}, \sqrt{\pi_{1,i}}, \sqrt{\pi_ {2,i}},\sqrt{\pi_{3,i}})\right)
 			\QQ(\sqrt{p_1}).
 		\end{eqnarray*}
 	
 	\item Assume that $t=0$, $s=1$ and $r\geq 1$. For all $m\geq 3$, we have
 \begin{enumerate}[\rm $\bullet$]
 	\item If $\exists i\in\{1, \dots, r\}$ such that  $ \left(\frac{2}{ l_i}\right)_4=1$, then
 	\begin{eqnarray*}
 		E(L_{m,d})= L_{m,d} \left( \prod_{i=1}^{r}\QQ(\sqrt{\pi_{1,i}}, \sqrt{\pi_ {2,i}},\sqrt{\pi_{3,i}})\right)
 		\left( \prod_{i=1}^{r-1}\QQ( \sqrt{l_i})\right).
 	\end{eqnarray*}
 	\item else,
 	\begin{eqnarray*}
 		E(L_{m,d})= L_{m,d} \left( \prod_{i=1}^{r}\QQ(\sqrt{l_i}, \sqrt{\pi_{1,i}}, \sqrt{\pi_ {2,i}},\sqrt{\pi_{3,i}})\right).	
 	\end{eqnarray*}
 \end{enumerate}
 		
 		\item Assume that $s=0$, $t=1$ and $r\geq 1$. For all $m\geq 3$, we have
 	\begin{enumerate}[\rm $\bullet$]
 		\item If $\exists i\in\{1, \dots, r\}$ such that  $ \left(\frac{2}{ l_i}\right)_4=-1$, then
 		\begin{eqnarray*}
 			E(L_{m,d})= L_{m,d} \left( \prod_{i=1}^{r}\QQ(\sqrt{\pi_{1,i}}, \sqrt{\pi_ {2,i}},\sqrt{\pi_{3,i}})\right)
 			\left( \prod_{i=1}^{r-1}\QQ( \sqrt{l_i})\right).
 		\end{eqnarray*}
 		\item else,
 		\begin{eqnarray*}
 			E(L_{m,d})= L_{m,d} \left( \prod_{i=1}^{r}\QQ(\sqrt{l_i}, \sqrt{\pi_{1,i}}, \sqrt{\pi_ {2,i}},\sqrt{\pi_{3,i}})\right).	
 		\end{eqnarray*}
 	\end{enumerate}	
 
 	\item Assume that $r=0$ and $s=t=1$. For all $m\geq 3$, we have
            $$	E(L_{m,d})= L_{m,d}(\sqrt{p_1})= L_{m,d}(\sqrt{q_1}).$$
            
 	\item  Assume that $r=s=0$ and $t=1$ or  $r=t=0$ and $s=1$. For all  $m\geq 3$, we have
 		$$	E(L_{m,d})= L_{m,d}.$$
 		
 			\item  Assume that $t=s=0$ and $r=1$. For all  $m\geq 3$, we have
 			$$ E(L_{m,d})= L_{m,d}(\sqrt{\pi_{1,1}},\sqrt{\pi_{2,1}}), $$
 			where $\pi_{1,1}=a_1+4ib_1$ and $\pi_{2,1}=e_1+4f_1\sqrt{2}$ such that $l_1=a_1^2+16b_1^2=e_1^2-32f_1^2.$

 		\end{enumerate}	
 		
 	\end{theorem}
 	
 \begin{proof}
 	To simplify notations let us put  $L=L_{3,d}=\mathbb{Q}(\zeta_{8}, \sqrt{d})=K( \sqrt{d})$, with $K =\mathbb{Q}(\zeta_{8})=\QQ(\sqrt{-1}, \sqrt{2})$.	
 	\begin{enumerate}[\rm 1.]
 		
 		\item As the couple $(x_i,y_i)$ is a primitive solution of
 		$p_1p_i=x^2+y^2$ $(i\geq 2)$ such that $x_i\equiv 1\pmod{2}$ and $y_i\equiv 0\pmod 4$, then  $\gamma_i=x_i+y_i\sqrt{-1}\equiv  x_i\pmod 4 \equiv \pm  1\pmod 4$. Since    $i=\sqrt{-1}\in L$, the equation $\gamma_i \equiv \xi^2\pmod{4}$ has solutions in $L$, furthermore the prime ideals of $K$ above $p_1$ and $p_i$ are ramified in  $L/K$, so $(\gamma_i)$ is the square of a fractional ideal of  $L$. It follows by Proposition \ref{unramifed quad ex}, that $\forall i\in\{2,...,n\}$, $L(\sqrt{\gamma_i})/L$ is an unramified extension  and it is very clear that the extension $L(\sqrt{p_i})/L$, $1\leqslant i \leqslant n-1$  is also unramified.\\
 		On the other hand by  \cite[Theorem 1]{Chemstherank}, we have $r_2(\mathbf{C}l(L))=2n-2$. Consider the following set :
 		$$\eta=\{p_1,...,p_{n-1},\gamma_2,..., \gamma_n \}.$$  
 		Let us show that the elements of  $\eta$  are linearly independents modulo   ${L^*}^2$. Put $\beta =\left(\displaystyle\prod_{i=1}^{n-1} {p_i}^{a_i}\right)\left(\displaystyle\prod_{j=2}^n {\gamma_j} ^{b_j}\right)$, where $a_i, b_j \in \{0,1\}$ are not all zero. Assume that $\beta \in {L^*}^2 .$ Thus, $\beta \in {K^*}^2 $ or $d\beta \in {K^*}^2 .$

 		\noindent Note that $\exists \, j \in \{2,3,\dots,n\}$ with $\,b_j\not=0$ (else, we get $\beta = \displaystyle\prod_ {i=1}^{n-1} {p_i}^{a_i} \in {K^*}^2 $, then $\sqrt{\beta}\in {K^*} $.
 		So  $\mathbb{Q}(\sqrt{\beta})$ is a quadratic subfield of 
 		$K=\mathbb{Q}(\zeta_8)=\QQ(\sqrt{-1}, \sqrt{2}) $. But this is impossible, since the only quadratic subfields of  $K=\mathbb{Q}(\zeta_8)$ are $k=\mathbb{Q}(\sqrt{-1})$, $k'=\mathbb{ Q}(\sqrt{2})$ and $k''=\mathbb{Q}(\sqrt{-2})$). We therefore distinguish two cases :
 		\begin{enumerate}[\rm $\bullet$]
  		\item If $\beta \in {K^*}^2 .$ Then 
 			$N_{K/k}( \gamma_j ) =x_j^2+y_j^2=p_1p_j $. Thus,
 			$$N_{K/k}(\beta) =\left(\displaystyle\prod_{i=1}^{n-1} {p_i}^{a_i}\right)^2\left(\displaystyle\prod_ {j=2}^n {p_1p_j}^{b_j}\right)=p_1^\ell\left(\displaystyle\prod_{j=2}^n {p_j}^{b_j}\right)Z^2, $$
 			such that $\ell\in\{0,1\}$ and $Z\in \mathbb Q$. This clearly implies that $p_1^\ell\left(\displaystyle\prod_{j=2}^n {p_j}^{b_j}\right)$ is a square in $k$, which is impossible.
 			Thus, $\beta \notin {K^*}^2 .$
 			
 			\item If $d\beta \in {K^*}^2,$ then  $N_{K/k}(d\beta)=d^2N_{K/k}(\beta)$. As above we show that this implies that $p_1^\ell\left(\displaystyle\prod_{j=2}^n {p_j}^{b_j}\right)\in k^2$, which is also  impossible.
 			So $\beta$ can not be a square in  $L$.  Therefore, the elements of  $\eta$ are linearly independents  modulo $L^2$. If follows that, $\eta$ is a representative set of  $\Delta(L)/L$.
 			Hence,   $$E(L_{3,d})=\mathbb{Q}(\zeta_{8}, \sqrt{d}, \sqrt{p_1}, ..., \sqrt{p_{n-1} },\sqrt{\gamma_2}...\sqrt{\gamma_n}).$$ 
 			Note that by \cite[Theorem 1]{Chemstherank}, we have  $r_2(\mathbf{C}l( L_{m,d}))=2n-2$, for all $m\geq 3$.
 			 Therefore, by Corollary \ref{useful cor}, we get :
 			
 			$$ E(L_{m,d})= L_{m,d}( \sqrt{p_1}, ..., \sqrt{p_{n-1}},\sqrt{\gamma_2},... ,\sqrt{\gamma_n}).$$
 		\end{enumerate}	
 		
 		\item By  \cite[Theorem 1]{Chemstherank}, we have $r_2(\mathbf{C}l(L))=2n-2$. Consider the set:
 		$$\eta=\{q_1,...,q_{n-1},\alpha_2,..., \alpha_n \},$$
 		Since  $\forall i\in\{2, \dots, n\}$:   $\alpha_i=x_i'+y_i'\sqrt{-2}$, where the couple $(x_i',y_i')$ is a primitive solution of
 		$q_1q_i=x^2+2y^2$ $(i\geq 2)$ such that $x_i'\equiv 1\pmod{2}$ and $y_i'\equiv 0\pmod 2$, then  $\alpha_i\equiv  x_i'+y_i'+\frac{-1+\sqrt{-2}}{2}2y_i'\pmod 4\equiv x_i'+y_i'\pmod 4  \equiv \pm  1\pmod 4.$
 	  	Proceeding as above we show that  the extensions  $L(\sqrt{\alpha_i})/L$ and $L(\sqrt{q_i})/L$ are unramified  and the elements of  $\eta$ are linearly independents  modulo $L^2$. Thus, $\eta$ is the set of generators of  $\Delta(L^*)/{L^*}^2$. Therefore, 
 		$$E(L)=L_{3,d} (\sqrt{q_1}, ..., \sqrt{q_{n-1}},\sqrt{\alpha_2},... ,\sqrt{\alpha_n}),$$   
 		and so by \cite[Theorem 1]{Chemstherank} and  Corollary \ref{useful cor}, we get :
 		$$ E(L_{m,d})= L_{m,d}( \sqrt{q_1}, ..., \sqrt{q_{n-1}},\sqrt{\alpha_2},... ,\sqrt{\alpha_n}).$$
 		
 		\item  By  \cite[Theorem 1]{Chemstherank}, we have  $r_2(\mathbf{C}l(L))=2n-3$. Put
 		$$\eta=\{{p_1}, ..., {p_{s}},{q_1}, ..., {q_{t-1}},{\gamma_2}, ...,{\gamma_s},{\alpha_2},...,{\alpha_t}\},$$
 		Let us show that the elements of    $\eta$ are linearly independents   modulo ${L^*}^2 $. Set
 		$\beta =\left(\displaystyle\prod_{i=1}^{s} {p_i}^{a_i}\right) \left(\displaystyle\prod_{i=1}^{t-1} {q_i}^{a_i'}\right)\left(\displaystyle\prod_{j=2}^s {\gamma_j}^{b_j}\right)\left(\displaystyle\prod_{j=2}^t {\alpha_j}^{b_j'}\right)$, where $a_i, b_j,a_i', b_j' \in \{0,1\}$  are not all zero. Assume that  $\beta \in {L^*}^2 .$ Then, $\beta \in {K^*}^2 $ or $d\beta \in {K^*}^2 .$ Note that $\exists j\in\{2, \dots, s\},\;b_j\neq 0$ or  $\exists j\in\{2, \dots, t\},\;b_j'\neq 0$.
 		
 		\begin{enumerate}[\rm $\bullet$]
 			\item Assume that $\beta \in {K^*}^2$. Since  
 			$N_{K/k''}( \alpha_j ) ={x_j'}^2+2{y_j'}^2=q_1q_j $ and $N_{K/k}( \gamma_j ) =x_j^2+y_j^2=p_1p_j $, we have
 			$$N_{K/k''}(\beta) =\left(\displaystyle\prod_{i=1}^{s} {p_i}^{a_i}\prod_{i=1}^{t-1 } {q_i}^{a_i'}\prod_{j=2}^s {\gamma_j}^{b_j}\right)^2\left(\displaystyle \prod_{j=2}^t {(q_1q_j)} ^{b_j'}\right)= q_1^{\ell }\left(\displaystyle \prod_{j=2}^t {q_j}^{b_j'} \right)Z_1^2,$$
 			$$N_{K/k}(\beta) =\left(\displaystyle\prod_{i=1}^{s-1} {p_i}^{a_i}\prod_{i=1}^{t- 1} {q_i}^{a_i'}\prod_{j=2}^s {\alpha_j}^{b_j'}\right)^2\left(\displaystyle \prod_{j=2}^t {(p_1p_j )}^{b_j}\right)= p_1^{\ell' }\left(\displaystyle \prod_{j=2}^t {p_j}^{b_j} \right)Z_2^2,$$
 			for some $\ell, \ell'\in\{0,1\}$ and $Z_i\in \mathbb Q$. This implies that $ q_1^{\ell }\left(\displaystyle\prod_{j=2}^t {q_j}^{b_j'} \right)$ is a square in  $k''$ and $ p_1^{\ell'}\left(\displaystyle\prod_{j=2}^t {p_j}^{b_j} \right)$  is a square in $k$, which is impossible. Thus, $\beta \notin {K^*}^2 .$
 			\item Analogously, we check that $d\beta \in {K^*}^2$ is impossible.
 		\end{enumerate}	
 		Since the extensions $L(\sqrt{p_i})/L$, $L(\sqrt{q_i})/L$, $L(\sqrt{\gamma_j})/L$ and $L(\sqrt{\alpha_j})/L$ are unramified, then the elements of  $\eta$ are generators of $\Delta(L^*)/{L^*}^2$. Hence, 
 		\begin{eqnarray*}
 			E(L_{m,d})&=& L_{m,d}( \sqrt{p_1}, ..., \sqrt{p_{s}}, \sqrt{q_1}, ..., \sqrt{q_{t-1}}, \sqrt{\gamma_2},... ,\sqrt{\gamma_s}, \sqrt{\alpha_2},... ,\sqrt{\alpha_t})\\
 			&=& L_{m,d}
 			\left( \prod_{i=1}^{s}\QQ(\sqrt{p_i})\right)	\left( \prod_{i=1}^{t-1}\QQ(\sqrt{q_i})\right)
 			\left( \prod_{i=2}^{s}\QQ(\sqrt{\gamma_i})\right) \left( \prod_{i=2}^{t}\QQ(\sqrt{\alpha_i})\right).
 		\end{eqnarray*}

 		\item In this case, we have $r_2(\mathbf{C}l(L))=n-3= 4r+2s+2t-3$ . Consider the set:
 		$$\eta=\{l_1,..., l_r,  {p_1}, ..., {p_{s}},{q_1}, ..., {q_{t-1}},{\pi_{1,1}},...,{\pi_{1,r}},\pi_{2,1},...,{\pi_{2,r}},\pi_{3,1},...,{\pi_{3,r}}, ,{\gamma_2},...,{\gamma_s},{\alpha_2},...,{\alpha_t}\},$$
 		
 		Put
 		$$\beta=\left(\displaystyle\prod_{i=1 }^{r}{l_i}^{\theta_1}{\pi_{1,i}}^{\theta_2}{\pi_{2,i}}^{\theta_3}{\pi_{3,i}}^{\theta_4}\right) \left(\displaystyle\prod_{i=1}^{s} {p_i}^{\theta_1'}\right) \left(\displaystyle\prod_{i=1}^{t-1} {q_i}^ {\theta_2'}\right)\left(\displaystyle\prod_{j=2}^s {\gamma_j}^{\theta_3'}\right)\left(\displaystyle\prod_{j=2}^t {\alpha_j}^{\theta_4'}\right),$$
 		
 		where ${\theta_{1,i}},{\theta_{2,i}},{\theta_{3,i}},{\theta_{4,i}},  {\theta_{1,i}'}, {\theta_{2,i}'}, {\theta_{3,i}'}, {\theta_{4,i}'} \in \{0,1\}$ are not all zero.\\ Assume that $\beta \in {L^*}^2 $, then   $\beta \in {K^*}^2 $ or  $d\beta \in {K^*}^2 $ (note that, as in the first item, the exponents  ${\theta_{2,i}},{\theta_{3,i}},{\theta_{4,i}}, {\theta_{3,i}'}, {\theta_{4,i}'}$ 
 		can not all vanish, since elsewhere we get   $\beta=\left(\displaystyle\prod_{i=1 }^{r}{l_i}^{\theta_{1,i}}\right)\left(\displaystyle\prod_{i=1}^{s} {p_i}^{\theta_{1,i}'}\right)\left(\displaystyle\prod_{i=1}^{t-1} {q_i}^ {\theta_{2,i}'}\right)$ is a square in  $K^*$, which implies that $\QQ(\sqrt{\beta})$ is a quadratic subfield of  $K$ which is impossible). 
 		\begin{enumerate}[\rm $\bullet$]
 			\item Let $\beta \in {K^*}^2 .$ Since, 	$N_{K/k}( \pi_{1,i} ) =a_i^2+16b_i^2=l_i $,  
 			$N_{K/k}( \gamma_j ) =x_j^2+y_j^2=p_1p_j $, $N_{K/k^{'}}( \pi_{2,i} ) =e_i^2-32b_i^2=l_i$,  $N_{K/k^{''}}( \pi_{3,i} ) =u_i^2+8v_i^2=l_i $ and 	$N_{K/k^{''}}( \alpha_j ) ={x_j'}^2+2{y_j'}^2=q_1q_j$, then

 			{\small
 				\begin{eqnarray*}
 					N_{K/k}(\beta) &=& \left(\displaystyle\prod_{i=1 }^{r}{l_i}^{\theta_{1,i}}{\pi_{2,i}}^{\theta_{3,i}}{\pi_{3,i}}^{\theta_{4,i}}\right)^2\left(\displaystyle\prod_{i=1 }^{r}{l_i}^{\theta_{2,i}}\right)\left(\displaystyle\prod_{i=1}^{s} {p_i}^{\theta_{1,i}'}\right)^2 \left(\displaystyle\prod_{i=1}^{t-1} {q_i}^ {\theta_{2,i}'}\right)^2  \left(\displaystyle\prod_{j=2}^s ({p_1p_j})^{\theta_{3,i}'}\right) \left(\displaystyle\prod_{j=2}^t {\alpha_j}^{\theta_{4,i}'}\right)^2    
 					\\
 					&=& p_1^{\ell }\left(\displaystyle\prod_{i=1 }^{r}{l_i}^{\theta_{2,i}}\right)\left(\displaystyle \prod_{j=2}^t {p_j}^{\theta_{3,i}'} \right)Z_1^2. 
 				\end{eqnarray*}
 				
 				\begin{eqnarray*}
 					N_{K/k^{'}}(\beta) &=& \left(\displaystyle\prod_{i=1 }^{r}{l_i}^{\theta_{1,i}}{\pi_{1,i}}^{\theta_{2,i}}{\pi_{3,i}}^{\theta_{4,i}}\right)^2\left(\displaystyle\prod_{i=1 }^{r}{l_i}^{\theta_{3,i}}\right)\left(\displaystyle\prod_{i=1}^{s} {p_i}^{\theta_{1,i}'}\right)^2 \left(\displaystyle\prod_{i=1}^{t-1} {q_i}^ {\theta_{2,i}'}\right)^2 \left(\displaystyle\prod_{j=2}^s {\gamma_j}^{\theta_{3,i}'}\right)^2 \left(\displaystyle\prod_{j=2}^t {\alpha_j}^{\theta_{4,i}'}\right)^2    
 					\\
 					&=&\left(\displaystyle\prod_{i=1 }^{r}{l_i}^{\theta_{3,i}}\right)Z_2^2. 
 				\end{eqnarray*}
 				
 				\begin{eqnarray*}
 					N_{K/k^{''}}(\beta) &=& \left(\displaystyle\prod_{i=1 }^{r}{l_i}^{\theta_{1,i}}{\pi_{1,i}}^{\theta_{2,i}}{\pi_{2,i}}^{\theta_{3,i}}\right)^2\left(\displaystyle\prod_{i=1 }^{r}{l_i}^{\theta_{4,i}}\right)\left(\displaystyle\prod_{i=1}^{s} {p_i}^{\theta_{1,i}'}\right)^2 \left(\displaystyle\prod_{i=1}^{t-1} {q_i}^ {\theta_{2,i}'}\right)^2\left(\displaystyle\prod_{j=2}^s {\gamma_j}^{\theta_{3,i}'}\right)^2 \left(\displaystyle\prod_{j=2}^t ({q_1q_j})^{\theta_{4,i}'}\right)    
 					\\
 					&=& q_1^{\ell^{'}} \left(\displaystyle\prod_{i=1 }^{r}{l_i}^{\theta_{4,i}}\right) \left(\displaystyle \prod_{j=2}^t {q_j}^{\theta_{4,i}'} \right)Z_3^2. 
 			\end{eqnarray*}}
 			
 			such that $\ell, \ell^{'}, \in\{0,1\}$ and $Z_i\in \mathbb Q$. Thus, $ p_1^{\ell }\left(\displaystyle\prod_{i=1 }^{r}{l_i}^{\theta_{2,i}}\right)\left(\displaystyle \prod_{j=2}^t {p_j}^{\theta_{3,i}'} \right)$ is a square in   $k$, $ \left(\displaystyle\prod_{i=1 }^{r}{l_i}^{\theta_{3,i}}\right)$ is a square in   $k'$ and $ q_1^{\ell^{'}} \left(\displaystyle\prod_{i=1 }^{r}{l_i}^{\theta_{4,i}}\right) \left(\displaystyle \prod_{j=2}^t {q_j}^{\theta_{4,i}'} \right) $ is a square in    $k^{''}$, then    the above three   equations are impossible. 
 			\item We similarly show that  $d\beta \in {K^*}^2$ is impossible. Thus, the elements of $\eta$ are linearly independents modulo ${L^*}^2 $.
 		\end{enumerate} 
 		On the other hand as above,  we check that the primes of   $L$ generated respectively by   $\pi_{1,i}$, ${\pi_{2,i}}$ and $\pi_{3,i}$ are the squares of certain fractional  ideals of  $L$ and as    $a_i\equiv e_i \equiv \pm 1\pmod 4$, since they are odd, the equations $\pi_{1,i} \equiv \xi^2\pmod 4$ and  $\pi_{2,i} \equiv \xi^2\pmod 4$ have solutions in  $L$ (since $i\in L$). Furthermore, we have $\pi_{3,i}=u_i+2v_i\sqrt{-2}=u_i+2v_i+4v_i \frac{-1+\sqrt{-2}}{2}\equiv u_i+2v_i\pmod 4\equiv \pm 1\pmod 4$ ($u_i$ is odd), then  the equation  $\pi_{3,i} \equiv \xi^2\pmod 4$ also has solutions $L$. Thus, by  Proposition \ref{unramifed quad ex}, the extensions  $L(\sqrt{\pi_{1,i}})/L$, $L(\sqrt{\pi_{2,i}})/L $ and $L(\sqrt{{\pi_{3,i}}})/L$  are unramified for  $i\in\{1,   \dots, k\}$. It follows that  $\eta$ is a  set of generators of $\Delta(L^*)/{L^*}^2$. Hence, we have the fourth item.
 		 
 		\item By  \cite[Theorem 1]{Chemstherank}, if $\exists i\in\{1, \dots, r\}$ such that  $ \left(\frac{2}{ l_i}\right)_4=-1$, then we have $r_2(\mathbf{C}l(L))=4r+2s-3$. Put:
 		$$\eta=\{l_1,..., l_{r-1},  {p_1}, ..., {p_{s}}, {\pi_{1,1}},...,{\pi_{1,r}},\pi_{2,1},...,{\pi_{2,r}},\pi_{3,1},...,{\pi_{3,r}}, ,{\gamma_2},...,{\gamma_s}\},$$				
 		As above we show that  $\eta$ is a set of generators of  $\Delta(L)/{L^*}^2$. Therefore,

 		\begin{eqnarray*}
 			E(L_{m,d})= L_{m,d} \left( \prod_{i=1}^{r}\QQ(\sqrt{\pi_{1,i}}, \sqrt{\pi_ {2,i}},\sqrt{\pi_{3,i}})\right)
 			\left( \prod_{i=1}^{r-1}\QQ( \sqrt{l_i})\right)
 			\left( \prod_{i=1}^{s-1}\QQ(\sqrt{p_i})\right)
 			\left( \prod_{i=2}^{s}\QQ(\sqrt{\gamma_i})\right).
 		\end{eqnarray*}
 		
 	\end{enumerate}	
 	With analogous reasoning,  we prove the rest of our theorem.
 \end{proof}

 	\section*{\bf Acknowledgment}
 The authors are so grateful to  the unknown referee  for his/her several helpful suggestions that helped us to improve our paper, and for calling our attention
 	to the missing details.

 \end{document}